\documentclass[12pt]{imsart}

\usepackage{latexsym,amsmath,amssymb,amscd,amsfonts,amsthm,mathrsfs,dsfont}
\usepackage[latin1]{inputenc}
\usepackage{natbib}
\usepackage{color}
\usepackage{geometry}
\usepackage{dsfont,float}
\usepackage{graphics,graphicx,epsfig,xspace}
\usepackage{lscape}

\theoremstyle{plain}

\newtheorem{thm}{Theorem}

\newtheorem{lem}{Lemma}

\def \F {\mathcal{F}}
\def \K {\mathcal{K}}
\def \R {\mathbb{R}}
\def \P {\mathbb{P}}
\def \E {\mathbb{E}}

\def \N {\mathcal{N}}

\newcommand{\beqe}{\begin{eqnarray*}}
\newcommand{\eeqe}{\end{eqnarray*}}
\newcommand{\pa}[1]{\left({#1}\right)}
\newcommand{\cro}[1]{\left[{#1}\right]}
\newcommand{\ab}[1]{\left|{#1}\right|}
\newcommand{\ac}[1]{\left\{{#1}\right\}}
\newcommand{\un}{\mathds{1}}

\begin{document}
\begin{frontmatter}

\title{Multidimensional two-component Gaussian mixtures detection}
\runtitle{Multidimensional two-component Gaussian mixtures detection}

\begin{aug}
\author{B\'eatrice Laurent \ead[label=e1]{beatrice.laurent@insa-toulouse.fr}}
\and
\author{Cl\'ement Marteau\ead[label=e2]{clement.marteau@insa-toulouse.fr}}
\and
\author{Cathy Maugis-Rabusseau\ead[label=e3]{cathy.maugis@insa-toulouse.fr}}
\address{Institut de Math\'ematiques de Toulouse, INSA de Toulouse, Universit\'e de Toulouse\\
         INSA de Toulouse,\\ 135, avenue de Rangueil,\\  31077 Toulouse Cedex 4, France.}
\runauthor{Laurent et al.}
\end{aug}

\begin{abstract}
Let $(X_1,\ldots,X_n)$ be a $d$-dimensional i.i.d sample from a distribution with density $f$. The problem of detection of a two-component mixture is considered. Our aim is to decide whether 
$f$ is the density of a standard Gaussian random $d$-vector ($f=\phi_d$) against $f$ is a two-component mixture: $f=(1-\varepsilon)\phi_d +\varepsilon \phi_d (.-\mu)$ where $(\varepsilon,\mu)$ are  unknown parameters. Optimal separation conditions on $\varepsilon, \mu, n$ and the dimension $d$ are established, allowing to separate both hypotheses with prescribed errors. Several testing procedures are proposed and two alternative subsets are considered. 
\end{abstract}



\end{frontmatter}

\section{Introduction}

Let $\underline X= (X_1,\ldots,X_n)$ be an i.i.d n-sample, where for all $i \in \lbrace 1,\dots, n \rbrace$, $X_i$ corresponds to a $d$-dimensional random vector, whose distribution admits a density $f$ w.r.t the Lebesgue measure on $\R^d$. In the following, we denote by $\phi_d(.)$ the density function of the standard Gaussian distribution $\N_d(0_d,I_d)$ on $\mathbb{R}^d$. Our aim is to test 
\begin{equation}
H_0 : f= \phi_d \quad  \mathrm{against} \quad  H_1 : f\in\F,
\label{eq:pb}
\end{equation}
where
$$
\F=\ac{f_{(\varepsilon,\mu)} : x\in\R^d \mapsto (1-\varepsilon) \phi_d(x) + \varepsilon \phi_d(x-\mu); \varepsilon\in]0,1[, \mu\in\R^d}
$$
is the set of two-component Gaussian mixtures on $\mathbb{R}^d$. Mixture models are at the core of several studies and provide a powerful paradigm that allows to model several practical phenomena. We refer to \cite{McLachlan_Peel} for an extended introduction to this topic.  

The particular case of a two-component mixture is sometimes referred as a contamination model. In some sense, a proportion $\varepsilon$ of the sample is driven from a (Gaussian) distribution centered in $\mu$ while the remaining part of the data is centered. In this context, the testing problem (\ref{eq:pb}) amounts to the detection of a plausible contamination inside the data at hand w.r.t. the null distribution.  We refer for instance to \cite{Donoho_Jin} for practical motivations regarding this problem. We stress that Gaussian mixture is at the core of our contribution since it provides a benchmark model for several practical applications. However, the results proposed in this paper could be certainly extended to a wide range of alternative distributions.  

In a unidimensional setting ($d=1$), the testing problem (\ref{eq:pb}) has been widely considered in the literature in the last two decades. A large attention has been payed to methods based on the likelihood ratio, see e.g. \cite{Chernoff_Lander}, \cite{Azais_Gassiat_Mercadier} or \cite{Garel}. Concerning the construction of optimal separation conditions on the parameters $(\varepsilon,\mu)$, we can mention the seminal contribution of \cite{Ingster}. These conditions have been reached by the higher-criticism procedure proposed in \cite{Donoho_Jin} in a specific \textit{sparse} context, i.e. when $\varepsilon \ll 1/\sqrt{n}$ as $n\rightarrow +\infty$. Then, several extensions of this contribution have been proposed in an extended context: we mention for instance \cite{Cai_Jin_Low} for a study including confidence sets and the \textit{dense} setting ($\varepsilon \gg 1/\sqrt{n}$ as $n\rightarrow +\infty$), \cite{Cai_Jeng_low} for heterogeneous and heteroscedastic mixtures, or \cite{Cai_Wu} where general distributions and separation conditions have been investigated. In a slightly different spirit, a procedure based on the order statistics and non-asymptotic investigations on the testing problem (\ref{eq:pb}) have been proposed in \cite{LLM2014}.

In the contributions mentioned above, only unidimensional distributions are considered. In a different setting (signal detection), multidimensional problems have been at the core of recent investigations. We mention e.g. \cite{ACCD_2011} or \cite{BI_2013} among others. In a recent paper, \cite{Verzelen_AriasCastro} address the problem of testing normality in a multidimensional framework. They consider two-component Gaussian mixture alternatives where the proportions  are fixed and the difference in means are sparse.  However, up to our knowledge, the  multidimensional testing problem as displayed in (\ref{eq:pb}) has never been studied so far. We stress that in our setting, the proportion $\varepsilon$ is allowed to depend on the number of observations $n$. The present paper proposes a first attempt in this context. Our aim is two-fold: we  establish \textit{optimal} separation conditions on the parameters $(\varepsilon,\mu)$ for the testing problem (\ref{eq:pb}) in a first time, and describe  the influence of the dimension $d$ on the corresponding problem. In the same time, we propose various testing procedures and compare their theoretical performances. 

In this paper, we  assume that the norm of the mean parameter $\mu$ is bounded on the alternative $H_1$. Given $M\in \mathbb{R}_+^*$, we  deal with the subsets $\mathcal{F}_2[M]$ and $\mathcal{F}_\infty[M]$ defined as
$$\mathcal F_2[M] = \ac{f_{(\varepsilon,\mu)} \in  \F;\ \varepsilon\in]0,1[, \|\mu\| \leq M},$$
and 
$$
\F_\infty[M] = \ac{ f_{(\varepsilon,\mu)} \in  \F;\ \varepsilon\in]0,1[, \|\mu\|_\infty \leq M },
$$
where for a given $\mu \in \mathbb{R}^d$,  $\|\mu\|=\pa{\sum_{j=1}^d \mu_j^2 }^{1/2}$ denotes the $l_2$-norm and $\|\mu\|_\infty = \max_{j=1\dots d} |\mu_j |$ corresponds to the $l_{\infty}$-norm. In this context, our main results can be gathered in the following theorem. 

\begin{thm}
\label{thm:main}
Let $\alpha,\beta \in ]0,1[$ be fixed and $\mathcal{C}_+=\mathcal{C}_{+}(\alpha,\beta,M)$, $\mathcal{C}_-=\mathcal{C}_{-}(\alpha,\beta,M)$ be two explicit constants.  Then, there exists a level-$\alpha$ testing procedure $\tilde\Psi_{\alpha,2}$ such that
$$ \underset{\underset{\varepsilon \|\mu\|> \mathcal{C}_+ d^{1/4}/\sqrt{n}}{f\in \mathcal F_2[M]}}{\sup} \P_f(\tilde\Psi_{\alpha,2}=0) \leq \beta\hspace*{0.5cm} \textrm { and }\hspace*{0.5cm}
\inf_{\Psi_\alpha} \underset{\underset{\varepsilon \|\mu\|> \mathcal{C}_- d^{1/4}/\sqrt{n}}{f\in \mathcal F_2[M]}}{\sup}\P_f(\Psi_{\alpha}=0) > \beta,$$
where the infimum is taken on all possible level-$\alpha$ testing procedures $\Psi_\alpha$. 

Similarly, there exists a level-$\alpha$ testing procedure $\tilde\Psi_{\alpha,\infty}$ such that
$$ \underset{\underset{\varepsilon \|\mu\|_\infty> \mathcal{C}_+ \sqrt{\ln(d)/ n}}{f\in \mathcal F_\infty[M]}}{\sup} \P_f(\tilde\Psi_{\alpha,\infty}=0) \leq \beta\hspace*{0.5cm} \textrm { and }\hspace*{0.5cm}
 \inf_{\Psi_\alpha} \underset{\underset{\varepsilon \|\mu\|_\infty> \mathcal{C}_- \sqrt{\ln(d) / n}}{f\in \mathcal F_\infty[M]}}{\sup}\P_f(\Psi_{\alpha}=0) > \beta.$$
\end{thm}

Theorem \ref{thm:main} indicates that the detection boundary associated to the testing problem (\ref{eq:pb}) for the alternative subset $\mathcal F_2[M]$ is of order $d^{1/4}/\sqrt{n}$: detection is impossible (with a prescribed level $\beta$) if $\varepsilon \|\mu\|$ is smaller than $d^{1/4}/\sqrt{n}$, up to some constant. 
In the case of the alternative set $\F_\infty[M]$, the detection boundary depends on $\sqrt{\ln(d)}$. In these two cases, we propose optimal testing strategies in Sections  \ref{ss:ub1}, \ref{ss:ub2} and \ref{ss:ub3}. \\

The paper is organized as follows. Two different lower bounds are proposed in Section~\ref{s:lb} for both subsets $\F_2[M] $ and $\F_\infty[M] $ and proved in Section~\ref{s:plb}. Associated upper bounds are established in Section~\ref{s:ub} and proved in Section~\ref{s:pub}. To this end, we will investigate the performances of three different testing procedures. Some useful lemmas are gathered in Appendix~\ref{s:append:1}, while Appendix~\ref{s:append:2} contains some technical results. 

All along the paper, we use the following notations.  For any density $g$ on $\mathbb{R}^d$, we denote respectively by $\P_g$ and $\E_g$ the probability and expectation under the assumption that the common density of each $X_i$ in the i.i.d. sample $(X_1, \ldots, X_n)$ is $g$. In the particular case where the $ X_1, \ldots, X_n$ are i.i.d. with common density $\phi_d$, which is associated to the null hypothesis  $H_0$, we write $\P_0:=\P_{\phi_d}$ and $\E_0:=\E_{\phi_d}$. A testing procedure $\Psi$  denotes a measurable function of the sample $\underline{X}$, having values in $\lbrace 0,1\rbrace$. By convention, we reject (resp. do not reject) $H_0$ is $\Psi=1$ (resp. $\Psi=0$). Given $\alpha \in ]0,1[$, the test $\Psi$ is said to be of level $\alpha$ if $\P_{0}(\Psi=1) \leq \alpha$. In such a case, we write $\Psi =\Psi_\alpha$.

\section{Lower bounds}
\label{s:lb}

%

\subsection{Lower bound for the alternative class $\F_2[M]$}
\label{ss:lb1}
The non asymptotic minimax separation rates have been introduced by \cite{Yannick}. Let us recall the main definitions. 
Given $\beta\in]0,1[$, the class of alternatives $\F_2[M]$ and a level-$\alpha$ test $\Psi_\alpha$, we define the uniform separation $\rho(\Psi_\alpha,\F_2[M],\beta)$ of $\Psi_\alpha$ over the class 
$\F_2[M]$ as the smallest positive number $\rho$ such that the test has a second kind error at most equal to $\beta$ for all alternatives $f_{(\varepsilon,\mu)}$ in $\F_2[M]$ such that $\varepsilon \|\mu\| \geq \rho$. 
More precisely,
$$
\rho(\Psi_\alpha,\F_2[M],\beta) = \inf \left\{\rho>0; \underset{\underset{\varepsilon \|\mu\| \geq \rho}{f\in \F_2[M]}}{\sup} \P_{f}(\Psi_\alpha=0)\leq \beta \right\}.
$$
Then, the $(\alpha,\beta)$-minimax separation rate over $\F_2[M]$ is defined as
$$
\underline{\rho}(\F_2[M],\alpha,\beta) = \underset{\Psi_\alpha}{\inf}\ \rho(\Psi_\alpha,\F_2[M],\beta),
$$
where the infimum is taken over all level-$\alpha$ tests $\Psi_\alpha$.\\

Theorem \ref{thm:lb1} proposes a lower bound for the minimax separation rate $\underline{\rho}(\mathcal F_2[M],\alpha,\beta)$.
The main ingredient for the proof (displayed in Section \ref{s:plb}) is the construction of particular distributions for which the separation of both hypotheses $H_0$ and $H_1$ will be impossible with a  prescribed level $\beta$.

\begin{thm}
\label{thm:lb1}
Let $\alpha,\beta \in ]0,1[$ such that $\alpha+\beta <0.29$. Define
$$
\rho^\# =   \frac{1}{2\sqrt{C(M)}} \frac{d^{1/4}}{\sqrt{n}}, \quad  \mathrm{where} \quad C(M)=1 + \frac{M^2}{2}e^{M^2}.
$$
Then, if $\rho<\rho^\#$,  
\begin{equation}
\underset{\Psi_\alpha}{\inf} \underset{\underset{\varepsilon \|\mu\|  \geq \rho}{f\in \mathcal F_2[M]}}{\sup} \P_f(\Psi_\alpha=0) >\beta,
\label{eq:lb1}
\end{equation}
where the infimum is taken over all level-$\alpha$ tests. In particular, this implies that
$$
\underline{\rho}(\mathcal F_2[M],\alpha,\beta) \geq \rho^\#.
$$
\end{thm}

Equation (\ref{eq:lb1}) indicates that the hypotheses $H_0$ and $H_1$ cannot be separated with prescribed first and second kind  errors $\alpha$ and $\beta$ following the value of the terms $\varepsilon, \|\mu \|, d$ and $n$. In particular, for any level-$\alpha$ testing procedure, one can find a distribution $ f\in \mathcal{F}_2[M]$ such that $\varepsilon \|\mu \| \geq \rho^\#$ and $\P_f(\Psi_\alpha=0) >\beta$. This result is obtained thanks to an assumption on the levels $\alpha$ and $\beta$. This assumption is essentially technical and could be removed thanks to additional technical algebra. 

The condition $\varepsilon \|\mu \| \gtrsim d^{1/4} / \sqrt{n}$ is quite informative. First of all, since $\| \mu  \|$ is bounded, the proportion parameter $\varepsilon$ should be at least of order $1/\sqrt{n}$. This condition is often characterized as the \textit{dense} regime in the literature. In the same time, the 'energy' $\| \mu \|$ should not be to small if one expects to detect a potential contamination in the sample. It is worth pointing out that Theorem \ref{thm:lb1} precisely quantifies the role played by the dimension $d$ of the problem at hand.  We will see in Section \ref{s:ub} that this lower bound is optimal, up to some constant.

\subsection{Lower bound for the alternative class $\F_\infty[M]$}
\label{ss:lb2}

In this section, we concentrate our attention on the alternative $\F = \F_\infty[M]$. 
As in Section~\ref{ss:lb1}, we consider the $(\alpha,\beta)$-minimax separation rate over $\F_\infty[M]$ defined as
$$
\underline{\rho}(\F_\infty[M],\alpha,\beta) = \underset{\Psi_\alpha}{\inf}\ \rho(\Psi_\alpha,\F_\infty[M],\beta),
$$
where the infimum is taken over all level-$\alpha$ tests $\Psi_\alpha$ and 
$$
\rho(\Psi_\alpha,\F_\infty[M],\beta) = \inf \left\{\rho>0; \underset{\underset{\varepsilon \|\mu\|_\infty \geq \rho}{f\in \F_\infty[M]}}{\sup} \P_{f}(\Psi_\alpha=0)\leq \beta \right\}.
$$

Theorem \ref{thm:lb2} provides a lower bound for the minimax separation rate in this context. The proof is postponed to Section \ref{s:plb}.

\begin{thm}\label{thm:lb2}
Let $\alpha, \beta \in ]0,1[$ such that $ \alpha + \beta < 1$. Let 
$$
\rho^\star = \sqrt{ \frac{1}{C(M)}  \frac 1 n \ln\cro{1 + d \eta(\alpha,\beta)^2}}
$$
where $C(M)=(1+M^2 e^{M^2}/2)$ and $ \eta(\alpha,\beta) = 2(1-\alpha-\beta)$. Then, if $\rho <  \rho^\star$, 
$$
\underset{\Psi_\alpha}{\inf}\ \underset{\underset{\varepsilon \|\mu\|_\infty \geq \rho }{f\in \mathcal F_\infty[M]}}{\sup}\ \P_f(\Psi_\alpha=0) >\beta.
$$
This implies that
$$
\underline{\rho}(\mathcal F_\infty[M],\alpha,\beta) \geq \rho^\star.
$$
\end{thm}

Theorem \ref{thm:lb2} indicates that the detection condition on the parameters $(\varepsilon,\mu)$ is affected by the change of the reference norm. In particular, the dependency w.r.t. the dimension of the data is in this context of order $\sqrt{\ln(d)}$ which makes in some sense the detection problem easier than the one considered in Section \ref{ss:lb1}.

\section{Upper bounds}
\label{s:ub}

In Section \ref{s:lb}, we have proposed lower bounds on the separation region for the testing problem (\ref{eq:pb}). In particular, we have proved that in some specific cases, related to the value of the parameters $(\varepsilon,\mu)$, testing is impossible, i.e. every level-$\alpha$ tests will be associated to a second kind error greater than a prescribed level $\beta$. 

The aim of this section is to complete this discussion with upper bounds on the separation region. We propose three different testing procedures and investigate their related performances. In particular, we prove that these procedures reach the lower bounds presented above for both alternatives $\mathcal{F}_2[M]$ and $\mathcal{F}_\infty[M]$.  

\subsection{First testing procedure for the alternative class $\F_2[M]$}
\label{ss:ub1}

In a first time, we propose a procedure in the case where the alternative is expressed through the $l_2$-norm of $\mu$.
This procedure is based on the fluctuations of the empirical mean of the data.  Intuitively, $\E_f[X] = \varepsilon \mu$ for all random vectors having density $f \in \F_2[M]$, while $\mathbb{E}[X]=0$ under $H_0$. In particular, if the empirical mean of the sample has a large norm, there is a chance that the data have been driven w.r.t. a density $f$ that belongs to  $\F_2[M]$.\\

More precisely, given $\underline{X}=(X_1,\dots, X_n)$, set $\bar X_n =  \frac 1 n \underset{i=1}{\stackrel{n}{\sum}} X_i$. Let $\upsilon_\alpha$ denote the $(1-\alpha)$ quantile of a chi-squared distribution with $d$ degrees of freedom and define the test $\Psi_{1,\alpha}$ as
\begin{equation}
\Psi_{1,\alpha} = \mathds{1}_{\lbrace \|\sqrt{n} \bar X_n\|^2 > \upsilon_\alpha \rbrace}.
\label{eq:test1}
\end{equation}
The following theorem investigates the performances of this test. 

\begin{thm}\label{thm:ub1}
Let $\alpha, \beta \in]0,1[$ be fixed. Then, the testing procedure $\Psi_{1,\alpha}$ introduced in (\ref{eq:test1}) is of level $\alpha$. Moreover, there exists a positive constant $C(\alpha,\beta,M)$ depending only on $\alpha$, $\beta$ and $M$ such that
$$
	\underset{\underset{\varepsilon \|\mu\| \geq \rho}{f\in \F_2[M]}}{\sup}\ 
	\P_{f}\left(\Psi_{1,\alpha} = 0 \right) \leq \beta,
$$
for all $\rho \in \R_+^*$ such that
$$
	\rho \geq C(\alpha,\beta,M) \frac{d^{ 1/4}}{\sqrt n}.
$$
\end{thm}

The above result indicates that the test $\Psi_{1,\alpha}$ is powerful as soon as $f \in \mathcal{F}_2[M]$ with $\varepsilon \|\mu\| \gtrsim d^{1/4} / \sqrt{n}$. According to the lower bound displayed in Theorem~\ref{thm:lb1}, it appears that the minimax detection frontier is of order $d^{1/4}/ \sqrt{n}$ up to a constant, i.e. there exist $\mathcal{C}_-$ and $\mathcal{C}_+$ such that
\begin{itemize}
\item the hypotheses $H_0$ and $H_1$ cannot be separated if $\varepsilon \|\mu\| \leq \mathcal{C}_- d^{1/4} / \sqrt{n} $,
\item there exists a level-$\alpha$ powerful test as soon as $\varepsilon \|\mu\| \geq \mathcal{C}_+ d^{1/4} / \sqrt{n} $.
\end{itemize}

These two assertions together provide the first part of Theorem \ref{thm:main}. We stress that we do not investigate the value of the optimal constant associated to this separation problem ($\mathcal{C}_-$ and $\mathcal{C}_+$ do not match). Such a study indeed requires advanced asymptotic tools \citep[see e.g.][]{IS_2003} and is outside the scope of the paper.

\subsection{Second testing procedure for the alternative class $\F_2[M]$}
\label{ss:ub2}
In a unidimensional context, \cite{LLM2014} have introduced a testing procedure based on the ordered statistics. Such variables can indeed provide valuable informations on the deviation of the sample w.r.t. a benchmark distribution (e.g. under $H_0$). Although the ordered statistics of the sample are not clearly defined in a multidimensional setting, we can still project the data on a given axis and apply the procedure detailed in  \cite{LLM2014}.\\

To this end, we split the sample $\underline{X}$ in two different parts $\underline{A}=(A_1,\dots,A_{n/2})$ and $\underline{Y}=(Y_1,\dots, Y_{n/2})$. 
For the sake of convenience, we assume without loss of generality that $n$ is even and write $n/2=n$ in the sequel. Set
\begin{equation}
v_n = \frac{\bar A_n}{\|\bar A_n \|} \ \mathrm{where} \ \bar A_n = \frac{1}{n} \sum_{i=1}^n A_i.
\label{eq:barA_n}
\end{equation}
The axis generated by the vector $v_n$ indicates the direction of the empirical mean of the sample $\underline{A}$. Under $H_1$, it provides an information on the direction where the contamination in the sample has occurred. Then, we project the remaining data $\underline{Y}$ on this axis. To this end, define
$$ Z_i ^{(v_n)}= \langle Y_i , v_n \rangle  \quad \forall i\in \lbrace 1,\dots, n \rbrace.$$
The corresponding sample $\underline{Z} ^{(v_n)}= (Z_1^{(v_n)},\dots, Z_n^{(v_n)})$ is unidimensional
and we remark that under $H_0$, $Y_i \sim \mathcal{N}_d(0_d,I_d)$, hence conditionally on $v_n$, $Z_1^{(v_n)},\dots, Z_n^{(v_n)}$ are i.i.d. standard Gaussian variables 
since $\| v_n \|=1$. Therefore, we can apply the test introduced in \cite{LLM2014}. Recall that this test is based on the ordered statistics $Z_{(1)} \leq \dots \leq Z_{(n)}$ of an i.i.d. sample $(Z_1,\dots, Z_n)$ of standard Gaussian variables. More formally, assume that $n\geq 2$ and consider the subset $\K_n$  of $\{1, 2, \ldots, n/2 \}$ defined by
\begin{equation}
 \K_n =\{2^j, 0\leq j\leq \cro{\log_2(n/2)} \}.
 \label{def:Kn}
 \end{equation}
The test statistics is then defined as
\begin{equation}
\Psi_{2,\alpha} = \sup_{k\in \K_n} \mathds{1}_{\lbrace Z^{(v_n)}_{(n-k+1)}> q_{\alpha_n,k} \rbrace},
\label{eq:test2}
\end{equation}
where for all $u\in]0,1[$, $q_{u,k}$ denotes the $(1-u)$ quantile of $Z_{(n-k+1)}$  and
\begin{equation}
\alpha_n = \sup\ac{u\in]0,1[,\ \P_0\pa{\exists k\in\K_n, Z_{(n-k+1)} < q_{u,k}}\leq \alpha}.
\label{def:alphan}
\end{equation}
In some sense, we proceed to a multiple testing approach. The subset $\mathcal{K}_n$ is related to the different orders that are included in the detection process: we do not use the whole ordered sample in order to enhance the performances of our test. The term $\alpha_n$ then corresponds to a correction of the level of each individual test  $\mathds{1}_{\lbrace Z^{(v_n)}_{(n-k+1)}> q_{\alpha,k} \rbrace}$ that guarantees a final first kind error of level $\alpha$. The quantiles $q_{\alpha,k}$ can be explicitly computed. We refer to \cite{LLM2014} for more details regarding the construction of this testing procedure. 

Theorem \ref{thm:ub2} enhances the performances of the test $\Psi_{2,\alpha}$.

\begin{thm}
\label{thm:ub2}
Let $\alpha, \beta \in]0,1[$ be fixed. Then, the testing procedure $\Psi_{2,\alpha}$ introduced in (\ref{eq:test2}) is of level $\alpha$. Moreover, there exists a positive constant $C(\alpha,\beta,M)$ depending only on $\alpha$, $\beta$ and $M$ such that
$$
	\underset{\underset{\varepsilon \|\mu\| \geq \rho}{f\in \F_2[M]}}{\sup}\ 
	\P_{f}\left(\Psi_{2,\alpha}  = 0 \right) \leq \beta,
$$
for all $\rho \in \R_+^*$ such that
$$
	\rho \geq \rho^\dagger := C(\alpha,\beta,M) \frac{d^{ 1/4}}{\sqrt n}\sqrt{\ln \ln n}
$$
provided that 
$$ {n} \varepsilon \geq C(\alpha,\beta,M). $$
\end{thm}

The main conclusion of the above result is that the test $\Psi_{2,\alpha}$ based on the ordered statistics has a maximal second  kind error bounded by $\beta$ as soon as $\rho \geq \rho^\dagger$. Hence, $\Psi_{2,\alpha}$ exhibits essentially the same level of performances than $\Psi_{1,\alpha}$. 
The only difference with the bound displayed in Theorem \ref{thm:ub1} is related to an additional log-term ($\sqrt{\ln \ln n}$). Indeed, we proceed to a multiple testing procedure since we consider several indices for the ordered statistics. This log-term corresponds to the price to pay for such a construction. However, we stress that it could be removed provided that the upper bound $M$ on $\|\mu\|$ is known. 

This testing procedure has been included in the present study since we expect some robustness properties w.r.t. futur extensions of this work (see Section \ref{s:conclusion} for an extended discussion).

\subsection{A testing procedure for the alternative class $\mathcal{F}_\infty[M]$}
\label{ss:ub3}

In the previous section, we have applied the procedure proposed in \cite{LLM2014} on the projection of the data on a given axis (generated by the empirical mean of the sample). Here, we alternatively consider a projection on the canonical axis. In other word, we apply the procedure of \cite{LLM2014} on each canonical direction in order to detect a possible contamination. As summarized in Theorem \ref{thm:ub3} below, this approach appears to be convenient when dealing with the alternative class $\mathcal{F}_\infty[M]$.

Let $(e_1,\ldots,e_d)$ be the canonical basis of $\R^d$. In the following, we set
$$Z_{ij} : = \langle X_i, e_j \rangle \quad \forall i\in \lbrace 1,\dots, n \rbrace, \ j\in \lbrace 1,\dots d \rbrace. $$ 
Then, for a given $j\in \lbrace 1, \dots, d \rbrace$, we can remark that $(Z_{i,j})_{i=1\dots n}$ is a unidimensional sample and we denote by $(Z_{(i),j})_{i=1\dots n}$ its associated ordered statistics. Then, for a given level $\alpha\in]0,1[$, we set 
$$
	T_{\alpha,j}^+  := \sup_{k \in \K_n} \left\{  \mathds{1}_{Z_{(n-k+1),j} > q_{\alpha_n,k}}  \right\},
$$
and
$$
	T_{\alpha,j}^-  := \sup_{k \in \K_n} \left\{  \mathds{1}_{Z_{(k),j} < - q_{\alpha_n,k}}  \right\},
$$
where $\K_n$, $\alpha_n$ and $q_{\alpha,k}$ are defined as in Equations~\eqref{def:Kn} and \eqref{def:alphan}.  
Then, we consider the following test statistics
\begin{equation}
\label{def:testdirection}
\Psi_{3,\alpha} = \underset{1\leq j \leq d}{\sup}\ \max\left(T_{\frac{\alpha}{2d},j}^+,T_{\frac{\alpha}{2d},j}^-\right).
\end{equation}
The following theorem provides a control on the first and second kind errors of the test $\Psi_{3,\alpha}$ when the alternative is measured via the infinite norm.

\begin{thm}
\label{thm:ub3}
Let $\alpha, \beta \in ]0,1[$ be fixed. Then, the test $\Psi_{3,\alpha}$ introduced in
\eqref{def:testdirection} is of level $\alpha$. Moreover, if $n \geq 3$ and
$$8.25\times \frac{\ln\cro{4 d \log_2(n/2)/\alpha}}{n} \leq \int_M^{+\infty} \phi_1(x) dx,$$
then, there exists a positive constant $ C(\beta,M) $ depending only on $\beta$ and M, such that
$$ \sup_{\underset{ \varepsilon \|\mu\|_\infty \geq \rho}{f\in\F_\infty[M]}}  \P_f(\Psi_{3,\alpha}=0)\leq \beta, $$
as soon as
\begin{equation}\label{Conddensedirection}
\rho \geq C(\beta,M) \times \sqrt{\frac{\ln\ln(n) \ln(d/\alpha)}{n}}.
\end{equation}
\end{thm}

Theorem \ref{thm:ub3}, together with Theorem \ref{thm:lb2} allow to characterize the separation frontier for the testing problem (\ref{eq:pb}) when the energy (norm of $\mu$) is measured w.r.t. the infinite norm. As discussed in Section \ref{s:lb}, the problem appears to be easier in this setting as the dimension of the data grows: the price to pay is a term of order $\sqrt{\ln(d)}$ (against $d^{1/4}$ with the $l_2$-norm). 

By the way, the construction of the test $\Psi_{3,\alpha}$ highlights the presence of this log-term: we proceed to $2d$ different tests (for each dimension, in the directions $e_j$ and $-e_j$), and reject $H_0$ as soon as one of these tests detects something. This exactly corresponds to a multiple testing approach. The price to pay relies in the Bonferroni correction in each test: $\alpha$ is replaced by $\alpha/2d$, which implies the presence of a $\sqrt{\ln(d)}$ term in the separation condition.

\section{Discussion}
\label{s:conclusion}

In our opinion, the main contribution of this paper is a sharp characterization of the role played by the dimension in a two-component mixture detection context. As discussed above, the price to pay in a multidimensional setting is a term of order $d^{1/4}$ (resp. $\sqrt{\ln(d)}$) when the energy in the alternative is measured w.r.t. the $l_2$-norm (resp. $l_\infty$-norm). At this step, several questions are still open and provide possible outcomes for futur investigations. 

First of all, according to the classical denomination in the statistical literature, our investigations have been drawn in a \textit{dense} regime. Indeed, the proportion parameter $\varepsilon$ is not allowed to be (asymptotically) smaller than $1/\sqrt{n}$. On the other hand, when $d=1$, several analyses have been conducted in a so-called \textit{sparse} regime, i.e. when $\varepsilon \ll 1/\sqrt{n}$. In this context, it could be challenging to investigate the testing problem (\ref{eq:pb}) in a \textit{sparse} context and to precisely determine the influence of the dimension $d$ on the problem. By the way, it seems necessary to propose a procedure that will be convenient for any considered regime, i.e. in some sense adaptive to the asymptotic of parameter $\varepsilon$. 

Several additional investigations could be driven in this setting, among them: considering more general benchmark distributions (i.e. different from the standard Gaussian distribution), heteroscedastic mixtures or taking into account some uncertainty on the reference distribution. All these questions are outside the scope of the paper but could be at the core of future contributions.

\section{Proof of Theorems~\ref{thm:lb1} and \ref{thm:lb2}}
\label{s:plb}
For the sake of convenience, we introduce the subset $\mathcal F[\rho,M]$ which corresponds to
$$\mathcal F_2[\rho,M] = \{f\in \mathcal F_2[M]; \varepsilon \|\mu\|  \geq  \rho\}$$
in the first proof, 
and 
$$\mathcal F_\infty[\rho,M] = \{f\in \mathcal F_\infty[M]; \varepsilon \|\mu\|_\infty  \geq  \rho\}$$
in the second proof,
for any given radius $\rho>0$. 
Following \cite{IS_2003} or \cite{Yannick}, we will use a Bayesian argument in order to bound the minimax separation radius in the two contexts.  
Thus, we consider a subset $\ac{g_\omega; \omega\in \Omega}$ of $\mathcal F[\rho,M]$ which will be specify for each proof later. Then,
$$
\sup_{f\in \mathcal F[\rho,M]} \P_f(\Psi_\alpha=0)  \geq   \P_{g_\omega}(\Psi_\alpha=0), \forall \omega\in\Omega.
$$
Denoting the uniform probability measure $\pi$ on the finite set  $\Omega$, we have
\begin{eqnarray}
\sup_{f\in \mathcal F[\rho,M]} \P_f(\Psi_\alpha=0) 
& \geq & \int_\Omega \P_{g_\omega}(\Psi_\alpha=0) d\pi(\omega) := \P_\pi (\Psi_\alpha=0) . \label{eq:Bayes}
\end{eqnarray}
Using \eqref{eq:Bayes} and similar computations as in \cite{IS_2003} or \cite{Yannick}, we obtain
\begin{eqnarray*}
\underset{\Psi_\alpha}{\inf} \underset{f \in \mathcal F[\rho,M]}{\sup} \P_{f}(\Psi_\alpha=0)& \geq &\underset{\Psi_\alpha}{\inf}\  \P_{\pi}(\Psi_\alpha =0) \\
&\geq& 1 - \alpha - \frac 1 2 \sqrt{\E_0[L_{\pi}^2(\underline X)] -1},
\end{eqnarray*}
where 
$L_{\pi}(\underline X) =\frac{d\P_{\pi}}{d\P_0} (\underline X)$ is the likelihood ratio. In particular, if we can ensure that
$$
 \E_0\cro{L_{\pi}^2(\underline X)  } < 1 + \eta(\alpha,\beta) ^2,
$$
where $\eta(\alpha,\beta) = 2( 1-\alpha - \beta)$ for all $\alpha,\beta\in ]0,1[$, then
$$\underset{\Psi_\alpha}{\inf} \underset{f \in \mathcal F[\rho,M]}{\sup} \P_{f}(\Psi_\alpha=0 )  > 1 - \alpha - \frac12 \eta(\alpha,\beta)  = \beta.$$
In the two following proofs displayed below, we will specify the subset $\ac{g_\omega; \omega\in \Omega}$ and propose an upper bound for the term $\E_0[L_{\pi}^2(\underline X)]$.

\subsection{Proof of Theorem~\ref{thm:lb1}}
In this proof, recall that
$$\mathcal F[\rho,M] = \mathcal F_2[\rho,M] = \{f\in \mathcal F_2[M]; \varepsilon \|\mu\|  \geq  \rho\},$$
for any given radius $\rho>0$. 

We now consider $r \in] 0,M] $ and $ \varepsilon \in ]0,1[$ such that $ \varepsilon r = \rho$. In this context, we choose $\Omega=\{-1,1\}^d$ and 
$$
\forall \omega\in\Omega,\ g_\omega(.) = (1-\varepsilon)\phi_d(.) + \varepsilon \phi_d\pa{. - \frac{r}{\sqrt{d}}\omega} \in \mathcal F_2[\rho,M].
$$
Then, we have to propose an upper bound for the term $\E_0[L_{\pi}^2(\underline X)]$ where in this setting
\begin{eqnarray*}
L_{\pi}(\underline X) 
&= &\frac{d\P_{\pi}}{d\P_0} (\underline X)\\
&=& \frac{1}{2^d}\sum_{\omega \in \ac{-1,1}^d} 
 \prod_{i=1}^n\cro{(1-\varepsilon)   + \varepsilon \frac{\phi_d(X_i - \frac{r}{\sqrt{d}} \omega )}{\phi_d(X_i)} } \\
 &=& \frac{1}{2^d}\sum_{\omega \in \ac{-1,1}^d} \prod_{i=1}^n\cro{(1-\varepsilon)   + \varepsilon e^{-\frac{r^2}{2}} e^{\langle X_i,  \frac{r}{\sqrt{d}} \omega \rangle }}.
 \end{eqnarray*}
Thus,
\begin{eqnarray*}
L_{\pi}^2(\underline X)  
& = & \frac{1}{2^{2d}}\sum_{\omega, \tilde{\omega} \in \ac{-1,1}^d} 
 \prod_{i=1}^n \left[(1-\varepsilon)^2   + \varepsilon(1-\varepsilon) e^{-\frac{r^2}{2}} \pa{e^{\langle X_i,  \frac{r}{\sqrt{d}} \omega \rangle } + e^{\langle X_i,  \frac{r}{\sqrt{d}} \tilde{\omega} \rangle }} \right.\\
 & & \hspace{8cm} \left.+ \varepsilon^2 e^{-r^2} e^{\langle X_i,  \frac{r}{\sqrt{d}} (\omega + \tilde{\omega}) \rangle }  \right].
 \end{eqnarray*}
Since for all $\mu \in \R^d$, $ \E_0 \cro{ e^{\langle X_i,  \mu \rangle }}=e^{\|\mu\|^2/2}$, we have
 
$$
\E_0[L_{\pi}^2(\underline X)] = \frac{1}{2^{2d}}\sum_{\omega, \tilde{\omega} \in \ac{-1,1}^d}  \prod_{i=1}^n\cro{(1-\varepsilon)^2   + 2 \varepsilon(1-\varepsilon)+ 
\varepsilon^2 e^{-r^2} e^{\frac{r^2}{2d} \|\omega + \tilde{\omega} \|^2}}.$$
Noticing that $ \|\omega + \tilde{\omega} \|^2 = 2d+2 \langle \omega, \tilde{\omega} \rangle$, 
\begin{eqnarray*}
\E_0[L_{\pi}^2(\underline X)] &=& \frac{1}{2^{2d}}\sum_{\omega, \tilde{\omega} \in \ac{-1,1}^d}  \prod_{i=1}^n\cro{1 - \varepsilon^2 + \varepsilon^2 e^{\frac{r^2}{d}
\langle \omega, \tilde{\omega} \rangle}} \\
&=& \frac{1}{2^{2d}}\sum_{\omega, \tilde{\omega} \in \ac{-1,1}^d} \cro{1+  \varepsilon^2 \pa{ e^{\frac{r^2}{d}
\langle \omega, \tilde{\omega} \rangle} -1 } }^n\\
&=& \E \cro{\ac{1+ \varepsilon^2 \pa{e^{\frac{r^2}{d}\langle W, \tilde{W}\rangle}-1} }^n }, 
\end{eqnarray*} 
where $W$ and $\tilde W$ are two independent $d$-dimensional Rademacher random variables, i.e. 
$$ \P(W= w ) = \P(\tilde W = w) = \frac{1}{2^d} \quad \forall w\in \lbrace -1,1 \rbrace^d.$$
Noticing that
$$ \langle W, \tilde{W}\rangle = \sum_{j=1}^d W_j \tilde{W}_j $$
and that the variables $ W_j \tilde{W}_j$ for $1\leq j \leq d$ are also i.i.d. Rademacher random variables, $  \langle W, \tilde{W}\rangle $ has the same distribution as $Y = \sum_{j=1}^d W_j$. This leads to 
$$ \E_0[L_{\pi}^2(\underline X)] = \E \cro{ \ac{1+\varepsilon^2 \pa{e^{\frac{r^2}{d}Y }-1}}^n}.$$
We now use the following inequality which holds for any real number $u$ such that $ \ab{u} \leq M$ : 
\begin{equation} \label{ineqexp}
 \ab{e^u-1-u} \leq \frac{e^M}{2}u^2.
 \end{equation}
Since $ \ab{  r^2  \frac{Y}{d} } \leq r^2 \leq M^2$, we have
$$ e^{\frac{r^2}{d} {Y }}-1 \leq \frac{r^2}{d}{Y }+ \frac{e^{M^2}}{2}  \frac{r^4}{d^2}Y^2.$$
Hence, we have
$$  0\leq 1- \varepsilon^2  \leq 1+ \varepsilon^2 \pa{e^{\frac{r^2}{d}Y }-1} \leq  1+ \frac{\varepsilon^2 r^2}{\sqrt{d}}  \pa{\frac{Y}{\sqrt{d} }+ \frac{e^{M^2}}{2}  M^2\frac{Y^2}{d} }. $$
Setting $C(M)= 1+  {e^{M^2}} M^2 /2,$
$$  0\leq 1+ \varepsilon^2 \pa{e^{\frac{r^2}{d}Y }-1} \leq  1+ C(M) \frac{\varepsilon^2 r^2}{\sqrt{d}}  \pa{\frac{\ab{Y}}{\sqrt{d} }\vee  \frac{Y^2}{d} }, $$
which leads to
$$  \E_0[L_{\pi}^2(\underline X)]  \leq \E \cro{\ac{  1+ a\pa{\frac{\ab{Y}}{\sqrt{d} }\vee  \frac{Y^2}{d} }}^n},$$
where 
\begin{equation}
a= C(M) {\varepsilon^2 r^2}/{\sqrt{d}}.
\label{eq:a}
\end{equation}
Using the inequality $\ln(1+x)\leq x$ for all $x \geq 0$, we have
\begin{eqnarray*}
\E_0[L_{\pi}^2(\underline X)]  &\leq & \E \cro{ e^{ \ac { n a \pa{\frac{\ab{Y}}{\sqrt{d} }\vee  \frac{Y^2}{d} }}}},\\
&\leq &  e^{ { n a } } \P \pa{  \frac{\ab{Y}}{\sqrt{d} } \leq 1}+ \E \cro{ e^{ { n a   \frac{Y^2}{d}} } \un_{\left\lbrace \frac{\ab{Y}}{\sqrt{d} } >1\right\rbrace}}.
\end{eqnarray*}
Moreover, using an integration by part
\begin{eqnarray*}
 \E \cro{ e^{ { n a   \frac{Y^2}{d}} } \un_{ \left\lbrace \frac{\ab{Y}}{\sqrt{d} } >1\right\rbrace}}  &\leq&  e^{ n a }  \P \pa{  \frac{\ab{Y}}{\sqrt{d} } > 1} + \int_{e^{ n a }}^{+\infty} \P \pa { e^{ n a   \frac{Y^2}{d}} >t } dt ,
\end{eqnarray*}
leading to 
$$ \E_0[L_{\pi}^2(\underline X)]  \leq e^{ n a }   +  \int_{e^{ n a }}^{+\infty} \P \pa {e^{ n a   \frac{Y^2}{d} } >t } dt. $$
We deduce from Hoeffding's inequality that for all $x >0$,
$$ \P \pa{ \frac{\ab{Y}}{\sqrt{d} } > x } \leq 2 \exp \pa{-x^2/2} .$$
Hence, for all $t > e^{ n a }$, 
$$ \P \pa {e^{ n a   \frac{Y^2}{d} } >t } \leq 2 t^{{-1}/{2na}}.$$
In the particular case where $na <1/2$,  we get
\begin{eqnarray*}
\E_0[L_{\pi}^2(\underline X)]  &\leq & e^{ n a } + 2  \int_{e^{ n a }}^{+\infty}  t^{{-1}/{2na}} dt, \\
 &\leq & e^{ n a } \pa{1    + \frac{4 na}{1-2 na} e^{  -1/2}} \leq h(na),
\end{eqnarray*}
where the function $h(.)$ is defined as 
$$h(x) = e^{ x } \pa{1    + \frac{4 x}{1-2 x} e^{  -1/2}} \quad \forall x \in [ 0, 1/2[ .  $$
The function $h$ is non decreasing on $ [ 0, 1/2[ $. Hence 
$$ na \leq 1/4 \Rightarrow \E_0[L_{\pi}^2(\underline X)]  \leq h(1/4) \leq 3< 1 + \eta(\alpha,\beta) ^2,$$
since,  according to our assumption
$$ \alpha +\beta < 1-\frac{1}{\sqrt{2}} \simeq 0.293 \Rightarrow (1-\alpha-\beta)^2 >1/2.$$
In order to conclude the proof, just remark from (\ref{eq:a}) that 
$$na \leq 1/4 \Leftrightarrow \varepsilon^2 r^2 \leq \sqrt{d}/(4C(M) n) .$$ 
Hence, setting $ (\rho^\#)^2 =  \sqrt{d}/(4C(M) n)$, we get that if $\rho <\rho^\#$, then
$ \E_0[L_{\pi}^2(\underline X)]   < 1 +\eta(\alpha,\beta) ^2$, which leads to the desired result. 

\begin{flushright}
$\Box$
\end{flushright}

\subsection{Proof of Theorem~\ref{thm:lb2}}
In this context,
$$\mathcal F[\rho,M] = \mathcal F_\infty [\rho,M] := \{f\in \mathcal F_\infty[M]; \varepsilon \|\mu\|_\infty  \geq  \rho\},$$ 
for any $\rho>0$. Let $r \in]0,M]$ such that $\varepsilon r = \rho $. In this context, we choose 
$$\Omega=\left\lbrace  \omega\in \lbrace 0,1 \rbrace^d \ s.t. \ \sum_{j=1}^d \omega_j = 1   \right\rbrace,$$
and we define for all $\omega\in\Omega$,
$$
g_\omega(.) = (1-\varepsilon) \phi_d(.) + \varepsilon \phi_d(. - r \omega ) \in \mathcal F_\infty [\rho,M].
$$

Now, we turn our attention to the control of the associated likelihood ratio. For each $j=1,\ldots,d$, let $D^{(j)}\in\{0,1\}^d$ such that $D^{(j)}_\ell = \mathds{1}_{\ell=j}$ and 
\begin{eqnarray*}
L_{\pi}(\underline{X})
&= &\frac{d\P_{\pi}}{d\P_0}(\underline X),\\
&=& \left[\prod_{i=1}^n \phi_d(X_i)\right]^{-1}  \left[ \frac 1 d \sum_{j=1}^d  \prod_{i=1}^n \left\{(1-\varepsilon) \phi_d(X_i) + \varepsilon \phi_d(X_i - r D^{(j)})\right\}\right],\\
&=& \frac 1 d \sum_{j=1}^d  U_j(\underline X),
\end{eqnarray*}
with 
$$ U_j(\underline X) = \prod_{i=1}^n \left\{(1-\varepsilon) + \varepsilon \frac{\phi_d(X_i - r D^{(j)})}{\phi_d(X_i)}\right\}.$$
Thus
\begin{equation}
\E_0[L_{\pi}^2(\underline X)] = \frac{1}{d^2} \sum_{j=1}^d  \E_0[U_j(\underline X)^2]  + \frac{1}{d^2} \sum_{k\neq j} \E_0[U_j(\underline X)U_k(\underline X)].
\label{eq:lr1}
\end{equation}
In a first time, we can remark that for all $j\in \lbrace 1,\dots, d \rbrace$
\begin{eqnarray*}
\E_0[U_j(\underline X)^2]
&=& \E_0\left[\left( \prod_{i=1}^n \left\{(1-\varepsilon) + \varepsilon \frac{\phi_d(X_i -r D^{(j)})}{\phi_d(X_i)} \right\} \right)^2\right],\\
&=& \E_{\phi_d}\left[\left\lbrace(1-\varepsilon) + \varepsilon \frac{\phi_d(X_1 - r D^{(j)})}{\phi_d(X_1)} \right\rbrace^2 \right]^n,
\end{eqnarray*}
and
\begin{eqnarray*}
\lefteqn{\E_{\phi_d}\left[\left\{(1-\varepsilon) + \varepsilon \frac{\phi_d(X - r D^{(j)})}{\phi_d(X)} \right\}^2 \right]}\\
&=& (1-\varepsilon)^2 + \varepsilon^2 \int_{\R^d}\frac{\phi_d^2(x - r D^{(j)})}{\phi_d(x)} dx +  2(1-\varepsilon)\varepsilon \int_{\R^d} \phi_d(x - r D^{(j)}) dx,\\
&=& (1-\varepsilon)^2 + \varepsilon^2 e^{r^2} +  2(1-\varepsilon)\varepsilon, \\
&=& 1 + \varepsilon^2 (e^{r^2}  -1),
\end{eqnarray*}
since $\int_{\R^d} \frac{\phi_d^2(x-\mu)}{\phi_d(x)} dx = \exp(\|\mu\|^2)$. Thus 
$$\E_0[U_j(\underline X)^2] = \left\{1 + \varepsilon^2 (e^{r^2}  -1)\right\}^n.$$
Concerning the second sum in (\ref{eq:lr1}), we obtain for all $j,k \in \lbrace 1,\dots, d \rbrace$, $j\neq k$ 
\begin{eqnarray*}
\lefteqn{\E_0[U_j(\underline X)U_k(\underline X)]}\\
& = & \E_0\left[ \prod_{i=1}^n \left\{(1-\varepsilon) + \varepsilon \frac{\phi_d(X_i - r D^{(j)})}{\phi_d(X_i)}\right\}  \left\{(1-\varepsilon) + \varepsilon \frac{\phi_d(X_i - 	rD^{(k)})}{\phi_d(X_i)}\right\}\right] ,\\
&=&\left\{\E_{\phi_d}\left[(1-\varepsilon)^2 +  (1-\varepsilon) \varepsilon \frac{\phi_d(X_1 - r D^{(j)}) + \phi_d(X_1 - r D^{(k)})}{\phi_d(X_1)}  \right. \right. ,\\
& & \hspace{5cm} \left. \left. + \varepsilon^2  \frac{\phi_d(X_1 - r D^{(j)}) \phi_d(X_1 - r D^{(k)})}{\phi_d(X_1)^2}  \right]\right\}^n,\\
&=& \left\{(1-\varepsilon)^2 + 2 (1-\varepsilon) \varepsilon  + \varepsilon^2 \exp[r^2 \langle D^{(j)},D^{(k)} \rangle ] \right\}^n,\\
&=& \left\{(1-\varepsilon)^2 + 2 (1-\varepsilon) \varepsilon  + \varepsilon^2  \right\}^n = 1,
\end{eqnarray*}
since $\int_{\R^d} \frac{\phi_d(x-\mu_1) \phi_d(x-\mu_2)}{\phi_d(x)} dx = \exp(\langle \mu_1,\mu_2 \rangle)$. Finally,
$$\E_0[L_{\pi}^2(\underline X)] = \frac{1}{d}  \left\{1 + \varepsilon^2 (e^{r^2} -1)\right\}^n  + \frac{d(d-1)}{d^2} .$$
We obtain
\begin{eqnarray*}
\E_0[L_{\pi}^2(\underline X)] < 1 + \eta(\alpha,\beta)^2
&\Leftrightarrow & \frac{1}{d}  \left\{1 + \varepsilon^2 (e^{r^2}-1)\right\}^n  + \frac{d(d-1)}{d^2}  < 1 + \eta(\alpha,\beta)^2,\\
&\Leftrightarrow & \left\{1 + \varepsilon^2 (e^{r^2}-1)\right\}^n < 1 + d \eta(\alpha,\beta)^2,\\
&\Leftrightarrow & \varepsilon^2 (e^{r^2}-1) <\exp\left[\frac 1 n \ln(1 + d \eta(\alpha,\beta)^2)\right] -1.
\end{eqnarray*}
Since $0<r \leq M$, 
$$\varepsilon^2  (e^{r^2}-1) \leq C(M) (\varepsilon r)^2 = C(M) (\rho^\star)^2,$$ 
where the constant $C(M)$ satisfies $C(M)=(1+M^2 e^{M^2}/2)$ (see \eqref{ineqexp}).  On the other hand, 
$$\exp\left[\frac 1 n \ln(1 + d \eta(\alpha,\beta)^2)\right] -1 > \frac 1 n \ln(1 + d \eta(\alpha,\beta)^2).$$
Thus, the condition 
$\E_0[L_{\pi}^2(\underline X)] <1 + \eta(\alpha,\beta)^2$ is fulfilled as soon as
$$
C(M)\rho^2 < \frac 1 n \ln(1 + d \eta(\alpha,\beta)^2) $$
which is equivalent to
$$ \rho < \sqrt{ \frac{1}{C(M)}  \frac 1 n \ln(1 + d \eta(\alpha,\beta)^2)}.
$$
This concludes the proof of Theorem~\ref{thm:lb2}. 
\begin{flushright}
$\Box$
\end{flushright}

\section{Proof of Theorems~\ref{thm:ub1}, \ref{thm:ub2} and \ref{thm:ub3}}
\label{s:pub}
\subsection{Proof of Theorem~\ref{thm:ub1}}
First, remark that 
$$ \| \sqrt{n} \bar X_n \|^2 = \sum_{j=1}^d \left( \frac{1}{\sqrt{n}} \sum_{i=1}^n X_{ij} \right)^2.$$ 
Under $H_0$, $X_{ij}$ are i.i.d. standard Gaussian random variables. Hence $\|\sqrt{n} \bar X_n\|^2$ is a chi-squared random variable with $d$ degrees of freedom   and  
$$\P_0(\Psi_{1,\alpha} =1) = \P_0 (\|\sqrt{n} \bar X_n\|^2 > \upsilon_\alpha)=\alpha,$$
according to the definition of the quantile $\upsilon_\alpha$.  The test $\Psi_{1,\alpha}$ is hence of level $\alpha$. \\

Now, we want to control the second kind error.  Under $H_1$, each variable $X_i$ can be written as 
$$X_i = V_i \mu + \eta_i,$$ 
where $V_i$ is a Bernoulli variable with parameter $\varepsilon$, $\eta_i\sim \N_d(0_d, I_d)$ and $V_i$ and $\eta_i$ are independent. Then
$$ \sqrt{n} \bar X_n = \frac{S}{\sqrt{n}} \mu + B$$
where $S=\underset{i=1}{\stackrel{n}{\sum}} V_i \sim \mathcal{B}(n,\varepsilon)$ is a binomial random variable with parameters $(n,\varepsilon)$, $B=\underset{i=1}{\stackrel{n}{\sum}} \eta_i  /{\sqrt{n}} \sim \N_d (0_d,I_d)$ and $S,B$ are independent. In particular, conditionally to $S$, the variable
$\|\sqrt{n} \bar X_n\|^2 = \left\| \frac{S}{\sqrt{n}} \mu + B\right\|^2$ has a non-central chi-squared distribution with $d$ 
degrees of freedom and noncentrality parameter $\lambda_S= \| \frac{S}{\sqrt{n}} \mu \|^2$. Introduce 
$$h_S = d + \lambda_S - 2 \sqrt{[d + 2 \lambda_S]\ln(2 /\beta)}.$$ 
According to Lemma~\ref{LemmaLaurent} in Appendix A \citep[see also][]{Laurent2012}
$$
  \P\left( \left\| \frac{S}{\sqrt{n}} \mu + B\right\|^2 \leq h_S \ \Big | \ S \right) \leq \frac{\beta}{2}.
$$
Hence, for each $f\in \F_2[M]$, 
\begin{eqnarray*}
\P_{f}\left(\Psi_{1,\alpha}=0\right) 
&=& \P_{f}\left(\|\sqrt{n} \bar X_n \|^2 \leq \upsilon_\alpha \right),\\
&=& \P\left(\left\|\frac{S}{\sqrt{n}} \mu + B\right\|^2 \leq \upsilon_\alpha \right),\\
&=& \P\left(\left\lbrace \left\|\frac{S}{\sqrt{n}} \mu + B\right\|^2 \leq \upsilon_\alpha \right\rbrace \cap \left\lbrace h_S\leq \upsilon_\alpha\right\rbrace \right) \\
& & \hspace{1cm} +\P\left(\left\lbrace \left\|\frac{S}{\sqrt{n}} \mu + B\right\|^2 \leq \upsilon_\alpha \right\rbrace \cap \left\lbrace h_S> \upsilon_\alpha \right\rbrace \right), \\
&\leq & \P\left(h_S \leq \upsilon_\alpha \right) + \P\left(\left\|\frac{S}{\sqrt{n}} \mu + B\right\|^2 \leq h_S \right),\\
&\leq & \P\left(h_S\leq \upsilon_\alpha\right) + \frac \beta 2.
\end{eqnarray*}
According to Lemma \ref{LemmaLaurentMassart} in Appendix A, 
$$
	\upsilon_\alpha \leq d+ b(\alpha,d)\quad  \textrm{ where }  \quad b(\alpha,d) = 2 \ln(1 / \alpha) + 2 \sqrt{d \ln(1 / \alpha)}.
$$
Hence
\begin{eqnarray*}
\P(h_S\leq \upsilon_\alpha) 
&\leq & \P(h_S \leq d + b(d,\alpha)),\\
&\leq & \P(\lambda_S - 2 \sqrt{[d + 2 \lambda_S]\ln(2 /\beta)} \leq b(d,\alpha)),\\
 &\leq & \P(\lambda_S - 2 \sqrt{2 \ln(2 /\beta)} \sqrt{\lambda_S} - [2 \sqrt{d \ln(2/ \beta)} +  b(d,\alpha)] \leq 0),\\
&\leq & \P(\sqrt{\lambda_S} \leq R(\alpha,\beta,d)),
\end{eqnarray*}
with 
$$R(\alpha,\beta,d) = \sqrt{2 \ln(2/\beta)} + \sqrt{2 \ln(2/\beta) + 2 \sqrt{d \ln(2/\beta)} + b(\alpha,d)} .$$ 
We notice that   $ R(\alpha,\beta,d) \leq C(\alpha,\beta) d^{1/4}$ where $C(\alpha,\beta)$ is a constant depending only on $\alpha$ and $\beta$. Assuming that $\sqrt{n} \varepsilon \|\mu\| > C(\alpha,\beta) d^{1/4}$ and  using a Tchebychev's inequality leads to 
\begin{eqnarray*}
\P(h_S\leq \upsilon_\alpha) 
&\leq & \P\left(S\leq \frac{\sqrt{n}}{\|\mu\|} C(\alpha,\beta) d^{1/4} \right),\\
&\leq & \P\left(|S - n\varepsilon| > n\varepsilon -\frac{\sqrt{n}}{\|\mu\|} C(\alpha,\beta) d^{1/4} \right),\\
&\leq & \frac{n \varepsilon \|\mu\|^2 }{[n\varepsilon \|\mu\| - \sqrt{n} C(\alpha,\beta) d^{1/4}]^2},\\
&\leq&  \frac{n \varepsilon \|\mu\| {M} }{[n\varepsilon \|\mu\| - \sqrt{n} C(\alpha,\beta) d^{1/4}]^2}.
\end{eqnarray*}
Then, $\P(h_S\leq \upsilon_\alpha)\leq \beta/2$ is satisfied if
$$ \frac{n \varepsilon \|\mu\|  }{[n\varepsilon \|\mu\| - \sqrt{n} C(\alpha,\beta) d^{1/4}]^2} \leq \frac{\beta}{2M}.$$
The last inequality is fulfilled if
$$
\varepsilon \|\mu\| \geq \frac{\sqrt{n} C(\alpha,\beta) d^{1/4}}{n} + \frac{M}{\beta n} + \frac{\sqrt{1 + 2 \sqrt{n} C(\alpha,\beta) d^{1/4} \beta/M}}{\beta n /M}
$$
noticing that  $\frac{nx}{(nx-c)^2}\leq B$ is fulfilled if and only if $x\notin\left[\frac c n + \frac{1}{2 B n} \pm \frac{\sqrt{1 + 4 c B}}{2 B n}\right]$.\\
Thus $\P(h_S\leq \upsilon_\alpha)\leq \beta/2$ if 
$$
\varepsilon \|\mu\| \geq \frac{\tilde C(\alpha,\beta,M) d^{1/4}}{\sqrt{n}} ,
$$
where $\tilde C(\alpha,\beta,M)$ is a positive constant depending on $\alpha, \beta $ and $M$.

\begin{flushright}
$\Box$
\end{flushright}

\subsection{Proof of Theorem~\ref{thm:ub2}}
Following the definition of $\alpha_n$, $\Psi_{2,\alpha}$ is ensured to be a level-$\alpha$ test. Then, under $H_1$, each variable $Y_i$ can be written as 
$$ Y_i = V_i \mu + \eta_i, \quad i=1\dots n,$$
where $V_i \sim \mathcal{B}(\varepsilon)$ denotes a random Bernoulli variable  and $\eta_i \sim \mathcal{N}_d(0_d,I_d)$, $V_i$ and $\eta_i$ being independent. Hence, for all $i\in \lbrace 1,\dots, n \rbrace$, we get
$$ Z_i^{(v_n)} = \langle Y_i , v_n \rangle = V_i \langle \mu, v_n \rangle + \langle \eta_i , v_n \rangle$$ 
which satisfies 
$$Z_i^{(v_n)} | v_n \sim (1-\varepsilon) \mathcal{N}(0,1) + \varepsilon \mathcal{N}(\langle \mu, v_n \rangle, 1). $$ 
Let $f\in\F_2[M]$ and $\Omega_{v_n}$ be the event defined as 
$$ \Omega_{v_n} = \left\lbrace \varepsilon \langle \mu, v_n \rangle > C\left(\frac{\beta}{2}, M\right)\frac{\kappa_n}{\sqrt{n}} \sqrt{\ln (1/\alpha) }\right\rbrace,$$
where $\kappa_n= \sqrt{\ln \ln (n) }$ and $ C(\beta, M)$ denotes the constant appearing in (\ref{Conddense}) (see Appendix B). In the following, the constant $C(\frac{\beta}{2}, M) \sqrt{\ln (1/\alpha) }$ is denoted $C_{\alpha,\beta, M}$.

Then, using Lemma \ref{lem:testunid} in Appendix B,  
we get 
\begin{eqnarray}
\P_f (\Psi_{2,\alpha} = 0 ) & \leq & \E_{v_n}\cro{  \E_f\cro{\mathds{1}_{\Psi_{2,\alpha}=0}  \mathds{1}_{\Omega_{v_n}}   | v_n }}  + \P_f( \Omega_{v_n}^c)\nonumber\\
&\leq& \E_{v_n}\cro{ \mathds{1}_{\Omega_{v_n}}  \E_f\cro{\mathds{1}_{\Psi_{2,\alpha}=0}    | v_n}}  + \P_f( \Omega_{v_n}^c)\nonumber\\
&\leq& \frac \beta 2 + \P_f( \Omega_{v_n}^c).
\label{eq:contr1}
\end{eqnarray}
Now, we turn our attention to the control of $\P_f( \Omega_{v_n}^c)$. Using the definition of $v_n$ we obtain
\begin{eqnarray*}
\P_f( \Omega_{v_n}^c)
& = & \P_f \left( \varepsilon \langle \mu,v_n \rangle \leq \frac{C_{\alpha,\beta,M}}{\sqrt{n}} \kappa_n \right), \\
& = & \P_f \left( \varepsilon \left\langle \mu,\frac{\bar A_n}{\| \bar A_n\|} \right\rangle \leq \frac{C_{\alpha,\beta,M}}{\sqrt{n}}\kappa_n \right), \\
& = & \P_f \left( \varepsilon \left\langle \mu,\bar A_n \right\rangle \leq \frac{C_{\alpha,\beta,M}}{n}\| \sqrt{n} \bar A_n\| \kappa_n \right).
\end{eqnarray*}
At this step, remark that the variable $\bar A_n$ can be written as
\begin{equation} 
\bar A_n = \frac{S}{n} \mu +  \frac{U}{\sqrt{n}},
\label{eq:An}
\end{equation}
where $S \sim \mathcal{B}(n,\varepsilon)$, $U\sim \mathcal{N}_d(0_d,I_d)$, and $S$ and $U$ are independent. Moreover, conditionally to $S$, $\| \sqrt{n} \bar A_n\|^2$ has a noncentral chi-squared distribution with $d$ degrees of freedom and noncentrality parameter $\lambda_S= \| \frac{S}{\sqrt{n}} \mu \|^2$. Introduce
$$
h(S) = d + \lambda_S+ 2 \sqrt{ (d+ 2 \lambda_S ) x_\beta } + 2 x_\beta
$$
with $x_\beta=\ln(4/\beta)$. According to Lemma \ref{LemmaLaurent}, 
$$
\P_f \left( \| \sqrt{n} \bar A_n\|^2 > h(S) \ | \ S \right) \leq \frac{\beta}{4}.
$$

Then,
\begin{eqnarray}
\P_f( \Omega_{v_n}^c) 
& = &  \P_f\left( \varepsilon \left\langle \mu,\bar A_n \right\rangle \leq \frac{C_{\alpha,\beta,M}}{n}\| \sqrt{n} \bar A_n\| \kappa_n \right),\nonumber \\
& \leq & \P_f \left( \varepsilon \left\langle \mu,\bar A_n \right\rangle \leq \frac{C_{\alpha,\beta,M}}{n} \sqrt{h(S)} \kappa_n \right) +  \P_f \left( \| \sqrt{n} \bar A_n\|^2 > h(S) \right), \nonumber\\
&\leq&  \P_f \left( \varepsilon \left\langle \mu,\bar A_n \right\rangle \leq \frac{C_{\alpha,\beta,M}}{n} \sqrt{h(S)} \kappa_n \right)  +  \frac{\beta}{4}\nonumber\\
&\leq&  \P_f \left( \frac{\varepsilon \|\mu\|^2}{n} S + \frac{\varepsilon \|\mu\|}{\sqrt{n}} Z \leq \frac{C_{\alpha,\beta,M}}{n} \sqrt{h(S)}\kappa_n \right)  +\frac{\beta}{4}\nonumber
\end{eqnarray}
since
\begin{equation} 
\varepsilon \langle \mu,\bar A_n \rangle = \frac{\varepsilon \|\mu\|^2}{n} S + \frac{\varepsilon \|\mu\|}{\sqrt{n}} Z,
\label{eq:An2}
\end{equation}
where $Z \sim  \mathcal{N}(0,1)$ and is independent of $S$. 

Let $C_{1,\beta} = \sqrt{{8}/{\beta}}$. According to the Tchebychev's inequality,
$$
\P\pa{|S - n\varepsilon | > C_{1,\beta} \sqrt{n\varepsilon} } \leq \frac{\beta}{8}.
$$
Thus, 
\begin{eqnarray*}
\P_f( \Omega_{v_n}^c) 
&\leq&  \P_f \left( \left\lbrace \frac{\varepsilon \|\mu\|^2}{n} S + \frac{\varepsilon \|\mu\|}{\sqrt{n}} Z \leq \frac{C_{\alpha,\beta,M}}{n} \sqrt{h(S)} \kappa_n \right\rbrace \cap \left\lbrace |S - n\varepsilon | \leq C_{1,\beta} \sqrt{n\varepsilon} \right\rbrace \right) \\
& & \hspace{9cm} +  \frac{\beta}{8} +\frac{\beta}{4}.\\
\end{eqnarray*}
Note that, if $|S - n\varepsilon | \leq C_{1,\beta} \sqrt{n\varepsilon}$ and $\sqrt{n\varepsilon} > 2C_{1,\beta}$ then $S\in\cro{\frac 1 2 n \varepsilon, \frac 3 2 n \varepsilon}$.  Hence, on the event $|S - n\varepsilon | \leq C_{1,\beta} \sqrt{n\varepsilon}$
\begin{eqnarray*}
\sqrt{h(S)}  
&=& \pa{d + \lambda_S+ 2 \sqrt{ (d+ 2 \lambda_S ) x_\beta } + 2 x_\beta}^{\frac 1 2}\\
&\leq& \sqrt{d + 2 x_\beta + 2 \sqrt{d x_\beta}}  + \frac{S \|\mu\|}{\sqrt{n}}  + \sqrt{\frac{2\sqrt{2 x_{\beta}} S \|\mu\|}{\sqrt{n}} }\\
&\leq& a_{1,\beta} \sqrt{d} + \frac 3 2 \sqrt{n} \varepsilon \|\mu\|  +  \sqrt{3 \sqrt{2 n x_{\beta}} \varepsilon \|\mu\|  } \\
&\leq& a_{1,\beta} \sqrt{d} + a_{2,\beta} \sqrt{n} \varepsilon \|\mu\| \\
\end{eqnarray*}
since $S\leq \frac 3 2 n \varepsilon$ and $\sqrt{n}\varepsilon \|\mu\| \geq 1$, where $ a_{1,\beta}$ and $a_{2,\beta}$ are two positive constants. Then, as soon as $\sqrt{n \varepsilon} > 2 C_{1,\beta}$,
\begin{eqnarray*}
\lefteqn{\P_f( \Omega_{v_n}^c) }\\
&\leq&  \P_f \left( \left\lbrace \frac{\varepsilon \|\mu\|^2}{n} S + \frac{\varepsilon \|\mu\|}{\sqrt{n}} Z \leq \frac{C_{\alpha,\beta,M}}{n}\kappa_n \sqrt{h(S)} \right\rbrace \cap  \left\lbrace |S - n\varepsilon | \leq C_{1,\beta} \sqrt{n\varepsilon} \right\rbrace \right)  + \frac{3\beta}{8} ,\\
&\leq& \P_f\pa{ 
\frac 1 2 \varepsilon^2 \|\mu\|^2  + \frac{\varepsilon \|\mu\|}{\sqrt{n}} Z 
\leq \frac{C_{\alpha,\beta,M}}{n} \kappa_n \cro{a_{1,\beta} \sqrt{d} + a_{2,\beta} \sqrt{n} \varepsilon \|\mu\|  }      }   + \frac{3\beta}{8}, \\
&\leq& \P_f\pa{ 
 \frac{\varepsilon \|\mu\|}{\sqrt{n}} Z 
\leq \frac{C_{\alpha,\beta,M}}{n} \kappa_n \cro{a_{1,\beta} \sqrt{d} + a_{2,\beta} \sqrt{n} \varepsilon \|\mu\|  } - \frac 1 2 \varepsilon^2 \|\mu\|^2   }   + \frac{3\beta}{8} .\\
\end{eqnarray*}
Let $z_{\beta/8}$ be the $\beta/8$ quantile of the standard Gaussian distribution. Then,
$$
\P\pa{\frac{\varepsilon \|\mu\|}{\sqrt{n}} Z \leq \frac{\varepsilon \|\mu\|}{\sqrt{n}} z_{\beta/8}} = \frac{\beta}{8}.
$$
Then, $\P_f( \Omega_{v_n}^c) \leq  \frac{\beta}{8} + \frac{3\beta}{8} = \frac \beta 2 $ if 
\begin{eqnarray}
&   &\frac{C_{\alpha,\beta,M}}{n} \kappa_n \cro{a_{1,\beta} \sqrt{d} + a_{2,\beta} \sqrt{n} \varepsilon \|\mu\|  } - \frac 1 2 \varepsilon^2 \|\mu\|^2  \leq \frac{\varepsilon \|\mu\|}{\sqrt{n}} z_{\beta/8} \nonumber \\
&\Leftrightarrow & 
n \varepsilon^2 \|\mu\|^2 -  2 C_{\alpha,\beta,M}\kappa_n \cro{a_{1,\beta} \sqrt{d} + a_{2,\beta} \sqrt{n} \varepsilon \|\mu\|  } + 2 z_{\beta/8} \sqrt{n}\varepsilon \|\mu\| \geq 0 \nonumber \\
&\Leftrightarrow & 
n \varepsilon^2 \|\mu\|^2 -  2 v_{1,n}  \sqrt{n} \varepsilon \|\mu\|  -  v_{2,n} \sqrt{d}   \geq 0,
\label{eq:ineq2}
\end{eqnarray}
with $v_{1,n} = C_{\alpha,\beta,M} \kappa_n a_{2,\beta} -z_{\beta/8} >0$ and $v_{2,n}=2 C_{\alpha,\beta,M} \kappa_n a_{1,\beta}>0$. 
Inequality \eqref{eq:ineq2} is fulfilled if 
$$
\sqrt{n} \varepsilon \|\mu\|  \geq v_{1,n} + \sqrt{v_{1,n}^2 + v_{2,n} \sqrt{d}}. 
$$

Finally, $\P_f\pa{\Psi_{2,\alpha} = 0}\leq \beta$ occurs as soon as 
\begin{equation}
\sqrt{n}\varepsilon \|\mu \| > C^* d^{1/4} \kappa_n \geq 1 \textrm{  and  } n \varepsilon \geq \tilde C
\label{eq:condDense}
\end{equation}
for some constants $C^*$ and $\tilde C$ which only depend on $\alpha, \beta$ and $M$. 

\begin{flushright}
$\Box$
\end{flushright}

\subsection{Proof of Theorem~\ref{thm:ub3}}
The test $\Psi_{3,\alpha}$ is a level-$\alpha$ test since
\begin{eqnarray*}
\P_0(\Psi_{3,\alpha}=1) &=&\P_0\left(\exists j\in\{1,\ldots,d\}; T_{\frac{\alpha}{2d},j}^+ = 1 \cup T_{\frac{\alpha}{2d},j}^- =1\right),\\
&\leq& \underset{j=1}{\stackrel{d}{\sum}} \left[  \P_0\left(T_{\frac{\alpha}{2d},j}^+ = 1\right) +  \P_0\left(T_{\frac{\alpha}{2d},j}^- =1\right)\right],\\
&\leq& d \times 2 \times \frac{\alpha}{2d} = \alpha.
\end{eqnarray*} 
We now consider the control on the second kind error of this test $\Psi_{3,\alpha}$ when the alternative is measured via the infinite norm. 
If 
$$
	\varepsilon \|\mu\|_\infty \geq C(\beta,M)  \times \sqrt{\frac{\ln\ln(n) \ln(2 d / \alpha)}{n}},$$ 
then there exists $j_0\in\{1,\ldots,d\}$ such that 
$$
\P_f\left( T_{\frac{\alpha}{2d},j_0}^+ = 0\right) \leq \beta \textrm{ or }   \P_f\left( T_{\frac{\alpha}{2d},j_0}^- = 0\right) \leq \beta
$$
according to Lemma \ref{lem:testunid} in Appendix B. Then,
\begin{eqnarray*}
\P_f(\Psi_{3,\alpha}=0) &=& \P_f \left( \underset{1\leq j \leq d}{\sup}\ \max\left(T_{\frac{\alpha}{2d},j}^+,T_{\frac{\alpha}{2d},j}^-\right) =0 \right),\\
& = & \P_f\left( \bigcap_{j=1..d} \lbrace  T_{\frac{\alpha}{2d},j}^+ = 0 \rbrace \cap \lbrace T_{\frac{\alpha}{2d},j}^- = 0\rbrace \right),\\
&\leq& \inf_{j=1\dots d} \left[\P_f(T_{\frac{\alpha}{2d},j}^+ = 0 )\wedge \P_f(T_{\frac{\alpha}{2d},j}^- = 0)\right
], \\
&\leq& \P_f(T_{\frac{\alpha}{2d},j_0}^+ = 0 )\wedge \P_f(T_{\frac{\alpha}{2d},j_0}^- = 0)\leq \beta.
\end{eqnarray*}

\begin{flushright}
$\Box$
\end{flushright}

\newpage
\appendix
\section{Properties for chi-squared distribution and noncentral chi-squared distribution.}
\label{s:append:1}

In this section, we present some well-known results there are useful throughout the proofs. The first lemma is concerned with deviation of a chi-squared random variable, proposed in \cite{LaurentMassart2000}.

\begin{lem}\label{LemmaLaurentMassart} 
Let $U$ be a chi-squared random variable with $d$ degrees of freedom. Then, 
\begin{itemize}
\item for any positive $x$,
$$
\left\{\begin{array}{l}
\P(U\geq d + 2 \sqrt{dx} + 2x)\leq e^{-x},\\
\P(U\leq d - 2 \sqrt{dx} )\leq e^{-x}.
\end{array}\right.
$$
\item For any given $\alpha\in]0,1[$, let $u(d,\alpha)$ be the $(1-\alpha)$-quantile of $\chi^2(d)$. Then
$$u(d,\alpha) \leq d + 2 \ln( 1 /\alpha) + 2 \sqrt{d \ln(1/ \alpha)} = d + b(\alpha,d).$$
\end{itemize}
\end{lem}

This second lemma provides the control of deviations of a noncentral chi-squared random variable, available in \cite{Birge2001}.

\begin{lem}\label{LemmaLaurent} 
Let $T$ be a noncentral chi-squared random variable with $d$ degrees of freedom and a noncentrality parameter $\lambda$. Then, for any positive $x$,
$$
\left\{\begin{array}{l}
\P\left(T \geq d + \lambda + 2 \sqrt{(d+2\lambda) x} + 2  x \right) \leq e^{-x},\\
\P\left(T \leq d + \lambda -2 \sqrt{(d+2\lambda) x}\right) \leq e^{-x}.
\end{array}\right.
$$
\end{lem}

\section{Unidimensional test}
\label{s:append:2}

We have previously proposed some testing procedures that uses results proposed in \cite{LLM2014} in a unidimensional context. For the sake of convenience, we reproduce a slightly different version of theses contributions in order to facilitate the understanding of the proofs. \\

Let $(Z_1,\ldots,Z_n)$ be i.i.d. random variables from an unknown density $g$ w.r.t. the Lebesgue measure
on $\mathbb{R}$. The order statistics are denoted by $Z_{(1)}\leq Z_{(2)}\leq \ldots\leq Z_{(n)}$. 
We want to test 
$$H_0:\  g = \phi(.) \quad \textrm{ versus } \quad H_1:\  g\in \mathcal G_1[M],$$ 
where $\phi(.)=\phi_1(.)$ is the unidimensional Gaussian density and 
$$
\mathcal G_1[M]=\{z\in\R \mapsto (1-\varepsilon) \phi(z) + \varepsilon \phi(z-\tau); 0 < \tau \leq M\}.
$$

Let $\alpha\in]0,1[$. 
Let $T_\alpha$ be the test statistics defined as
\begin{equation} 
\label{def:testStatOrdunid}
T_{\alpha} := \sup_{k \in \K_n} \left\{  \mathds{1}_{Z_{(n-k+1)} > z_{\alpha_n,k}}  \right\},
\end{equation}
where, for all $u \in ]0,1[$, $z_{u,k}$ is the $(1-u)$-quantile of $Z_{(n-k+1)} $ under the null hypothesis and
$$\alpha_n=\sup \left\{u \in ]0,1[, \P_{H_0}\left(\exists k \in \K_n, Z_{(n-k+1)} > z_{u,k}\right) \leq \alpha\right\}.$$
The following lemma establishes sufficient conditions that allow to control the second kind error of $T_\alpha$.

\begin{lem}
\label{lem:testunid}
Let $\beta\in]0,1-\alpha[$. Assume that $n \geq 3$ and
$$8.25\times \frac{\ln(2\log_2(n/2)/\alpha)}{n} \leq \int_M^{+\infty} \phi(x) dx.$$
Then, there exists a positive constant $ C(\beta,M) $ depending only on $\beta$ and M, such that if
\begin{equation}\label{Conddense}
\rho \geq C(\beta,M) \sqrt{\frac{\ln\ln(n)  \ln(1/\alpha)}{n}},
\end{equation}
then,
$$ \sup_{\underset{\varepsilon \tau \geq \rho}{g\in\mathcal G_{1}[M]}}  \P_g(T_{\alpha}=0)\leq \beta. $$
\end{lem}

\begin{proof}
By definition, the test statistics $T_{\alpha}$ introduced in \eqref{def:testStatOrdunid} is exactly of level $\alpha$, namely
$$ \P_{H_0}(T_\alpha = 1) = \P_{H_0}\pa{\exists k \in \K_n, Z_{(n-k+1)} > z_{\alpha_n,k}} \leq \alpha,$$
thanks to the definition of $\alpha_n$.

In order to  control the second kind error of the test $T_\alpha$,  we first give an upper bound for $z_{\alpha_n,k}$. 
Let $\bar\Phi(x) = 1 - \Phi(x)$, where $\Phi$ is the cumulative distribution function associated to the density function $\phi$. For all $\alpha \in ]0,1[$ and $k \in \{1,2, \ldots, n/2 \}$, let $t_{\alpha,k}$
be the positive real number defined as
\begin{equation}
\label{def:talphak}
\bar\Phi\left(t_{\alpha,k}\right) = \frac k n \left[1-\sqrt{\frac{2\ln(\frac 2 \alpha)}{k}}\right]
\end{equation}
if $k>2\ln(\frac 2 \alpha)$, and $t_{\alpha,k}=+\infty$ otherwise. According to Lemma A.1 in \citet{LLM2014}, $z_{\alpha_n,k} \leq t_{\alpha_n,k}$. Considering $g\in \mathcal G_1[M]$, we want to control the second kind error of the test:
\begin{eqnarray}
\P_g(T_\alpha =0) &=& \P_g\pa{\forall k \in \K_n,Z_{(n-k+1)}\leq z_{\alpha_n,k}}\nonumber \\
&\leq & \inf_{k \in \K_n}\P_g\pa{Z_{(n-k+1)} \leq z_{\alpha_n,k}}  
\label{adapt}.
\end{eqnarray}
Since $ z_{\alpha_n,k} \leq t_{\alpha_n,k} $, using Markov's inequality
\begin{eqnarray*}
\P_g\pa{Z_{(n-k+1)} \leq z_{\alpha_n,k}} 
&\leq& \P_g(Z_{(n-k+1)} \leq t_{\alpha_n,k})\\
&\leq&  \P_g\left(\sum_{i=1}^n \left\{\mathds{1}_{\{Z_i\leq t_{\alpha_n,k}\}} - q_1\right\}> n (1-q_1) -k\right) \\
&\leq& \frac{n(1-q_1)}{[n(1-q_1) -k]^2}
\end{eqnarray*}
if $n(1-q_1)-k>0$, where
$$
1- q_1 = \P_g\left(Z_1\geq t_{\alpha_n,k}\right)  
= (1-\varepsilon) \bar\Phi\left(t_{\alpha_n,k}\right) + \varepsilon \bar\Phi\left(t_{\alpha_n,k} - \tau\right).
$$
Note that the inequality $\frac{nx}{(nx-k)^2} \leq \beta$ is fulfilled if and only if $x\notin\left[\frac k n + \frac{2}{\beta n} \pm \frac{2\sqrt{1 + 4k\beta}}{\beta n}\right]$.
Then, 
$$
\P_f\left(Z_{(n-k+1)} < t_{\alpha_n,k}\right) \leq \beta,
$$
if 
\begin{equation}
(1-\varepsilon) \bar\Phi\left(t_{\alpha_n,k}\right) + \varepsilon \bar\Phi\left(t_{\alpha_n,k} - \tau\right)
> \frac k n + \frac{2}{n\beta} +  \frac{2\sqrt{1 + 4k\beta}}{n\beta} .
\label{Cond1:unid}
\end{equation}
Now, we consider
$k \in \K_n$ such that
$$ \frac{0.99}{2} \bar\Phi(M)\leq \frac{k}{n} \leq  0.99 \ \bar\Phi(M) .$$
The set of solutions of this inequation is not empty since under the assumptions of Lemma \ref{lem:testunid}, $  0.99 \ \bar\Phi(M) n \geq 1$.
Note that
$|\K_n| \leq \log_2(n/2)$, hence $\alpha _n \geq \alpha/|\K_n| \geq
\alpha/\log_2(n/2)$. We will show that Condition \eqref{Cond1:unid} is fulfilled.
Using a Taylor expansion at the order 1, 
$$
(1-\varepsilon) \bar\Phi\left(t_{\alpha_n,k}\right) + \varepsilon \bar\Phi\left(t_{\alpha_n,k} - \tau\right)
= \bar\Phi\left(t_{\alpha_n,k}\right) + \varepsilon \tau \phi(a)
$$
where $a$ belongs to the interval $\left]t_{\alpha_n,k}-\tau,t_{\alpha_n,k}\right[$.
We recall that $\bar\Phi\left(t_{\alpha_n,k}\right) = \frac{k}{n}  \left[1 - \sqrt{\frac{2 \ln( 2 /\alpha_n)}{k}}\right] $.
Using (\ref{def:talphak}),  we just have to show that
\begin{equation}\label{eq:anbn}
\varepsilon \tau \phi(a) \geq\frac{2}{\beta n} + \frac{2\sqrt{1 + 4k\beta}}{\beta n} + \frac{\sqrt{k}}{n} \sqrt{2\ln(2 /\alpha_n)},
\end{equation}
in order to prove that  Condition \eqref{Cond1:unid} holds. \\
\ \\
Next, we want to prove that $\left[t_{\alpha_n,k}-\tau, t_{\alpha_n,k}\right]$
remains included in a fixed interval $[c_1(M),c_2(M)]$ with $c_1(M)>0$.
\\
On one hand, we have
$$ t_{\alpha_n,k} \geq  \bar\Phi^{-1}\left( \frac{k}{n} \right) \geq \bar\Phi^{-1}\left(  0.99 \ \bar\Phi(M) \right),   $$
and
$$ t_{\alpha_n,k}  -M  \geq  \bar\Phi^{-1}\left(   0.99\  \bar\Phi(M) \right) -M :=c_1(M)>0.$$
Moreover,
\begin{eqnarray*}
  \bar\Phi\left(t_{\alpha_n,k}\right) &\geq & \frac{0.99}{2} \bar\Phi(M) - \sqrt{\frac{2 \ln( 2 /\alpha_n)}{n}}\sqrt{ 0.99 \  \bar\Phi(M) } \\
&\geq& \frac{\bar\Phi(M)}{200}
\end{eqnarray*}
since  $(8.25){\ln(2\log_2(n/2)/\alpha)}/{n} \leq \bar\Phi(M)$. This implies that
$$t_{\alpha_n,k} \leq  \bar\Phi^{-1}\left( \bar\Phi(M)/200\right)  :=c_2(M).$$
 Finally, since $\phi(a) \geq \phi(c_2(M))$,   \eqref{eq:anbn} is satisfied if
$ \varepsilon \tau \geq C(\beta,M) \sqrt{\ln\ln(n)  \ln(1/\alpha) / n} $ for some suitable constant $C(\alpha, \beta,M)$.

\end{proof}

\section*{Acknowledgements}
This work was supported by the French Agence Nationale de la Recherche (ANR-13-JS01-0001-01, project MixStatSeq).

\bibliographystyle{apalike}
\bibliography{biblio}
\end{document}